\documentclass{amsart}
\usepackage[cp850]{inputenc}

\newtheorem{theorem}{Theorem}[section]
\newtheorem{lemma}[theorem]{Lemma}

\newtheorem{corollary}[theorem]{Corollary}

\theoremstyle{remark}

\newcommand{\bin}[2]{\mbox{$\left(#1 \atop #2\right)$}}

\newcommand{\grw}{\mbox{$-I\times_fM^n$}}
\newcommand{\pmin}{\mbox{$p_{\mathrm{min}}$}}
\newcommand{\pmax}{\mbox{$p_{\mathrm{max}}$}}
\newcommand{\hmin}{\mbox{$h_{\mathrm{min}}$}}
\newcommand{\hmax}{\mbox{$h_{\mathrm{max}}$}}

\newcommand{\Ro}{\mbox{$\overline{R}$}}
\newcommand{\Rico}{\mbox{$\overline{\mathrm{Ric}}$}}
\newcommand{\m}{\mbox{$\Sigma$}}
\newcommand{\R}[1]{\mbox{${\mathbb R}^{#1}$}}

\newcommand{\g}[2]{\mbox{$\langle #1 ,#2 \rangle$}}
\newcommand{\fle}{\mbox{$\rightarrow$}}

\newcommand{\rf}[1]{\mbox{(\ref{#1})}}
\newcommand{\rl}[1]{{~\ref{#1}}}
\newcommand{\nablao}{\mbox{$\overline{\nabla}$}}
\newcommand{\xm}{\mbox{$\mathcal{X}(\m)$}}

\def\beq{\begin{equation}}
\def\eeq{\end{equation}}

\newcommand{\x}{\mbox{$\psi:\Sigma^n\rightarrow\grw$}}
\newcommand{\xdos}{\mbox{$\psi:\Sigma^2\rightarrow -I\times_fM^2$}}
\begin{document}
\title[Uniqueness of spacelike hypersurfaces in GRW spacetimes]
{Uniqueness  of spacelike hypersurfaces with constant higher order mean curvature in
generalized Robertson-Walker spacetimes}

\author{Luis J. Al\'\i as}
\address{Departamento de Matem\'{a}ticas, Universidad de Murcia, E-30100 Espinardo, Murcia,
Spain} \email{ljalias@um.es}
\thanks{L.J. Al\'\i as was partially supported by MEC/FEDER Grant MTM2004-04934-C04-02 and
Fundaci\'{o}n S\'{e}neca, Spain.}

\author{A. Gervasio Colares}
\address{Departamento de Matem\'{a}tica, Universidade Federal do Cear,  Campus do Pici,
60455-760 Fortaleza-Ce, Brazil} \email{gcolares@mat.ufc.br}
\thanks{A.G. Colares was partially supported by FUNCAP and PRONEX, Brazil}





\begin{abstract}
In this paper we study the problem of uniqueness for spacelike hypersurfaces with constant higher order mean curvature in
generalized Robertson-Walker (GRW) spacetimes. In particular, we consider the following question: {Under what conditions must a compact spacelike hypersurface with constant higher order mean curvature in a spatially closed GRW spacetime be a spacelike slice ?} We prove that this happens, esentially, under the so called \textit{null convergence condition}. Our approach is based on the use of the Newton transformations (and their associated differential operators) and the Minkowski formulae for spacelike hypersurfaces.
\end{abstract}

\maketitle

\section{Introduction}
\label{s1}
The mathematical interest for the study of spacelike hypersurfaces in spacetimes began in the seventies, with the works of Calabi \cite{Ca}, Cheng and Yau \cite{CY}, Brill and Flaherty \cite{BF}, Choquet-Bruhat \cite{CB1,CB2}, and later on with some other authors (cf. e.g. \cite{Go,St,BS,Ge,B,BCM}).
Moreover, the study of such hypersurfaces is also of interest from a physical point of view, because of its relation to several problems in general relativity (see, for instance, \cite{CFM,CBY,MT}).
A basic question on this subject is the uniqueness of spacelike hypersurfaces with constant mean curvature in certain spacetimes (we refer the reader to the introductions of the papers \cite{ARS1} and \cite{ABC} for a brief account of it). More recently, there has been also an increasing interest in the study of uniqueness of spacelike hypersurfaces with constant higher order mean curvature, including the case of the second order mean curvature which is closely related to the intrinsic scalar curvature of the hypersurface (see Section\rl{s2}).

Conformally stationary (CS) spacetimes are time-orientable spacetimes which are equipped with a globally defined timelike conformal vector field, and they include, for instance, the family of generalized Robertson-Walker (GRW) spacetimes. By a GRW spacetime, we mean a Lorentzian warped product \grw\ with
Riemannian fibre $M^n$ and warping function $f$. In particular, when the Riemannian factor $M^n$ has constant sectional curvature then \grw\ is classically called a Robertson-Walker (RW) spacetime (for the details, see Section\rl{s2}). In a GRW spacetime, the vector field given by $K(t,x)=f(t)(\partial/\partial_t)_{(t,x)}$ defines globally a timelike conformal field, which is also closed, in the sense that its metrically equivalent 1-form is closed. As observed by Montiel in \cite{Mo}, every CS spacetime which is equipped with a closed timelike conformal vector field is locally isometric to a GRW.

In this paper, we are interested in the study of uniqueness of compact spacelike hypersurfaces with constant higher order mean curvature in GRW spacetimes. First of all, recall that if a GRW spacetime \grw\ admits a compact spacelike hypersurface, then it must be spatially closed, in the sense that the Riemannian factor $M^n$ must be compact
\cite[Proposition 3.2 (i)]{ARS1}. On the other hand, observe that, for a spatially closed GRW spacetime \grw, the family of slices $M^n_t=\{t\}\times M^n$ constitutes a foliation of \grw\ by compact totally umbilical leaves with constant mean curvature $H(t)=f'(t)/f(t)$ and, more generally, constant $k$-th mean curvature $H_k(t)=(f'(t)/f(t))^k$ (for the details, see \cite{ARS1}). Therefore, it is natural to address the following question:

\begin{quote}
\textit{Under what conditions must a compact spacelike hypersurface with constant higher order mean curvature in a spatially closed GRW spacetime be a spacelike slice $M^n_t=\{t\}\times M^n$ ?}
\end{quote}

In \cite{ARS1}, the first author, jointly with Romero and S\'{a}nchez, considered that question for the case of constant mean curvature hypersurfaces. In particular, they found that, when the ambient space obeys the so called \textit{timelike convergence condition}, then every compact spacelike hypersurface with constant mean curvature must be totally umbilical and, except in very exceptional cases, it must be a spacelike slice. As observed by those same authors in \cite{ARS2}, the weaker \textit{null convergence condition} is enough to guarantee that the hypersurface must be totally umbilical, not only for hypersurfaces into GRW spacetimes but, more generally, for hypersurfaces into CS spacetimes equipped with a timelike conformal vector field which is an eigenfield of the Ricci operator (and, in particular, for CS spacetimes
which are equipped with a closed timelike conformal vector field). Later on, Montiel \cite{Mo} considered again that question and, after a careful classification of the totally umbilical hypersurfaces with constant mean curvature, he completed the program by showing that the only compact spacelike hypersurfaces with constant mean curvature into a GRW spacetime which satisfies the null convergence condition are the spacelike slices, unless in the case where the spacetime is a de Sitter space and the hypersurface is a round umbilical sphere.

With respect to the case of hypersurfaces with constant higher order mean curvature, in \cite{Mo} Montiel also obtained a uniqueness result for hypersurfaces with constant scalar curvature (equivalently, constant second order mean curvature) into CS spacetimes with constant sectional curvature. More recently, the authors, jointly with Brasil Jr., initiated the study of hypersurfaces with constant higher order mean curvature in CS spacetimes \cite{ABC}.

In all the references above considering that question, the main tool used to extract information about the spacelike hypersurfaces and prove the results are the so called Minkowski formulae. Actually, the use of Minkowski-type integral formulae in this context was first started by Montiel in \cite{Mo1} for the study of hypersurfaces with constant mean curvature in de Sitter space, and it was continued later by Al\'\i as, Romero and S\'{a}nchez in \cite{ARS1,ARS2,ARS3} for constant mean curvature hypersurfaces in GRW spacetimes and, more generally, in CS spacetimes. Let us remark here that for the case of the mean curvature, only the first and second Minkowski formulae come into play. However, for the case of higher
order mean curvatures, one needs to use successive Minkowski formulae, which involve covariant derivatives of the Ricci tensor of the ambient spacetime. That makes the task much harder, unless one assumes that the ambient spacetime is Einstein or, even more, it has constant sectional curvature. That is the case, for instance, in \cite{AAR} where higher order Minkowski formulae are developed for spacelike hypersurfaces in de Sitter, in \cite{Mo} where the third Minkowski formula is written for spacelike hypersurfaces in a CS spacetime with constant sectional curvature, and, more generally, in \cite{ABC} where the authors, jointly with Brasil Jr., derived general Minkowski formulae for spacelike hypersurfaces in CS spacetimes with constant sectional curvature.

In this paper, which is a natural continuation of the references above, we go deeper into the study of spacelike hypersurfaces with constant higher order mean curvature in spatially closed GRW spacetimes. Our approach, which is based in the use of the so called Newton transformations $P_k$ and their associated second order differential operators $L_k$ (see Section\rl{s3}), allows us to extend the previous uniqueness results to the case where the ambient spacetimes do not have constant sectional curvature. For instance, and as a first application, in Section\rl{s5} we obtain the following (Theorem\rl{t1}):
\begin{quote}
\begin{it}
Let \grw\ a spatially closed GRW spacetime such that its warping function satisfies
\[
ff''-f'^2\leq 0,
\]
(that is, such that $-\log f$ is convex). Let $\m^n$ be a compact spacelike
hypersurface immersed into \grw\ whose $k$-th Newton transformation $P_k$
is definite on \m, for some $k=0,1\ldots, n-1$. If the quotient $H_{k+1}/H_k$ is constant, then the
hypersurface is an embedded slice $\{t_0\}\times M$, where $t_0\in I$ satisfies $f'(t_0)\neq0$ if $k\geq 1$.
\end{it}
\end{quote}
The proof of Theorem\rl{t1} is based strongly on the ellipticity of the differential operator $L_k$ (or, equivalently, the definiteness of $P_k$; for the details, see Section\rl{s3}). In Corollary\rl{co1} and Corollary\rl{co2} we apply the result above to situations where the ellipticity of $L_k$ can be derived from geometric hypothesis.

On the other hand, in Section\rl{s6} we derive general Minkowski formulae
for hypersurfaces in GRW spacetimes which can be applied to the study of
hypersurfaces with constant higher order mean curvature in arbitrary RW
spacetimes, even if the ambient spacetime does not have constant sectional
curvature. Our derivation of Minkowski formulae in this generality is
based on the use of the operators $P_k$ and $L_k$, and on the detailed
analysis of a series of terms which come into play when the ambient space
does not have constant sectional curvature. In particular, we are able to
find a neat expression for the divergence of $P_k$ (see formula
\rf{divPKfinal}) which allows us to write our Minkowski formulae in a nice
way.

As a first application of these Minkowski formulae, in Section\rl{s7} we
obtain the following uniqueness result (Theorem\rl{th}) under the
assumption of null convergence condition (let us recall that a spacetime
obeys the null convergence condition if its Ricci curvature is
non-negative on null (or lightlike) directions; obviously, every spacetime
with constant sectional curvature trivially obeys null convergence
condition):

\begin{quote}
\begin{it}
Let \grw\ be a spatially closed RW spacetime obeying the null convergence condition, with $n\geq 3$, that is
\beq
\label{NCCbis2}
\kappa\geq\sup_I(ff''-f'^2),
\eeq
where $\kappa$ is the constant sectional curvature of $M^n$. Then every compact spacelike hypersurface immersed into \grw\ with positive constant $H_2$ is totally umbilical. Moreover, \m\ must be a slice $\{t_0\}\times M^n$ (necessarily with $f'(t_0)\neq 0$), unless in the case where \grw\ has positive constant sectional curvature and \m\ is a round umbilical hypersphere. The latter case cannot occur if the inequality in \rf{NCCbis2} is strict.
\end{it}
\end{quote}
See also Theorem\rl{thbis} and Theorem\rl{thbiss} for two different versions of Theorem\rl{th} for the general case of hypersurfaces with constant higher order mean curvature $H_k$, when $k\geq 3$.

Finally, in Section\rl{s9} and as another application of the second order differential operators associated to
the Newton transformations, we obtain the following uniqueness result for hypersurfaces in GRW spacetimes (Theorem\rl{thfinal}):
\begin{quote}
\begin{it}
Let \grw\ be a spatially closed GRW spacetime obeying the strong null convergence condition, with $n\geq 3$, that is,
\beq
\label{NCCstrongbis}
K_M\geq\sup_I(ff''-f'^2),
\eeq
where $K_M$ stands for the sectional curvature of $M^n$. Assume that $\m^n$ is a compact spacelike hypersurface immersed into \grw\ which is contained in a slab
\[
\Omega(t_1,t_2)=(t_1,t_2)\times M
\]
on which $f'$ does not vanish. If $H_{k}$ is constant, with $2\leq k\leq n$ then \m\ is totally umbilical. Moreover,
\m\ must be a slice $\{t_0\}\times M^n$ (necessarily with $f'(t_0)\neq 0$), unless in the case where \grw\ has positive
constant sectional curvature and \m\ is a round umbilical hypersphere.
The latter case cannot occur if we assume that the inequality in \rf{NCCstrongbis} is strict.
\end{it}
\end{quote}
The proof of Theorem\rl{thfinal} is based strongly on the ellipticity of the differential operators associated to the Newton transformations.

\section{Preliminaries}
\label{s2} Consider $M^n$ an $n$-dimensional Riemannian manifold,
and let $I$ be a 1-dimensional manifold (either a circle or an open
interval of \R{}). We will denote by \grw\ the $(n+1)$-dimensional product
manifold $I\times M$ endowed with the Lorentzian metric
\[
\g{}{}=-dt^2+f^2(t)\g{}{}_{M},
\]
where $f>0$ is a positive smooth function on $I$, and $\g{}{}_M$
stands for the Riemannian metric on $M^n$. That is, \grw\ is
nothing but a Lorentzian warped product with Lorentzian base
$(I,-dt^2)$, Riemannian fiber $(M^n,\g{}{}_M)$, and
warping function $f$. Following the terminology introduced in \cite{ARS1},
we will refer to \grw\ as a generalized Robertson-Walker (GRW) spacetime. In particular, when the Riemannian
factor $M^n$ has constant sectional curvature then \grw\ is classically called a Robertson-Walker (RW) spacetime.

Consider a smooth immersion \x\ of an $n$-dimensional connected manifold
\m\ into a GRW spacetime, and assume that the induced metric via $\psi$ is a Riemannian metric
on \m; that is, \m\ is a spacelike hypersurface. In that case,
since
\[
\partial_t=(\partial/\partial_t)_{(t,x)}, \quad (t,x)\in\grw,
\]
is a unitary timelike vector field globally defined on the ambient GRW spacetime, then there exists a unique unitary timelike normal field $N$
globally defined on \m\ which is in the same time-orientation as
$\partial_t$, so that
\[
\g{N}{\partial_t}\leq -1<0 \quad \mathrm{on} \quad \m.
\]
We will refer to that normal field $N$ as the future-pointing Gauss map of the hypersurface. Its opposite will be refered as the past-pointing Gauss map of \m.

Let $A:\xm\rightarrow\xm$ stand for the shape operator (or Weingarten
endomorphism) of \m\ with respect to either the future or the past-pointing Gauss map $N$. As
is well known, $A$ defines a self-adjoint linear operator on each tangent
space $T_p\m$, and its eigenvalues $\kappa_1(p), \ldots, \kappa_n(p)$ are
the principal curvatures of the hypersurface. Associated to the shape
operator there are $n$ algebraic invariants given by
\[
S_k(p)=\sigma_k(\kappa_1(p), \ldots, \kappa_n(p)), \quad 1\leq k\leq n,
\]
where $\sigma_k:\mathbb R^n\to\mathbb R$ is the elementary symmetric
function in $\mathbb R^n$ given by
\[
\sigma_k(x_1,\ldots, x_n)=\sum_{i_1<\cdots<i_k}x_{i_1}\ldots x_{i_k}.
\]
Observe that the characteristic polynomial of $A$ can be writen in terms
of the $S_k$'s as \beq \label{poly}
\det(tI-A)=\sum_{k=0}^n(-1)^kS_kt^{n-k}, \eeq where $S_0=1$ by definition.
The $k$th -mean curvature $H_k$ of the hypersurface is then defined by
\[
\bin{n}{k}H_k=(-1)^kS_k=\sigma_k(-\kappa_1,\ldots,-\kappa_n),
\]
for every $0\leq k\leq n$. In particular, when $k=1$
\[
H_1=-\frac{1}{n}\sum_{i=1}^n\kappa_i=-\frac{1}{n}\mathrm{trace}(A)=H
\]
is nothing but the mean curvature of \m, which is the main extrinsic
curvature of the hypersurface. The choice of the sign $(-1)^k$ in our
definition of $H_k$ is motivated by the fact that in that case the mean
curvature vector is given by $\overrightarrow{H}=HN$. Therefore, $H(p)>0$
at a point $p\in\m$ if and only if $\overrightarrow{H}(p)$ is in the same time-orientation as $N$.
Obviously, when $k$ is even the sign of $H_k$ does not depend on the chosen Gauss map. Even more, when $k$ is even it follows from the Gauss equation of the hypersurface that the value of $H_k$ is a geometric quantity which is related to the intrinsic curvature of \m. For instance, for $k=2$ one gets that
\beq
\label{scalar}
n(n-1)H_2=\bar{S}-S+2\Rico(N,N),
\eeq
where $S$ and $\bar{S}$ are, respectively, the scalar curvature of \m\ and \grw, and $\Rico$ stands for the Ricci tensor of the ambient GRW spacetime.

\section{The Newton transformations and their associated differential operators}
\label{s3}
The classical Newton transformations $P_k:\xm\rightarrow\xm$ are defined
inductively from $A$ by
\[
P_0=I \quad \mathrm{and} \quad P_k=\bin{n}{k}H_kI+A\circ P_{k-1},
\]
for every $k=1\ldots,n$, where $I$ denotes the identity in \xm.
Equivalently,
\[
P_k=
\sum_{j=0}^k \bin{n}{j}H_j A^{k-j}.
\]
Note that by the Cayley-Hamilton theorem, we have $P_n=0$ from \rf{poly}. Observe also that when $k$ is even, the definition of $P_k$ does not depend on the choosen Gauss map, but when $k$ is odd there is a change of sign in the definition of $P_k$.

Let us recall that each $P_k(p)$ is also a self-adjoint linear operator on
each tangent space $T_p\m$ which commutes with $A(p)$. Indeed, $A(p)$ and
$P_k(p)$ can be simultaneously diagonalized: if $\{ e_1, \ldots, e_n\}$
are the eigenvectors of $A(p)$ corresponding to the eigenvalues
$\kappa_1(p), \ldots, \kappa_n(p)$, respectively, then they are also the
eigenvectors of $P_k(p)$ with corresponding eigenvalues given by
\[
\mu_{i,k}(p)=(-1)^k\frac{\partial \sigma_{k+1}}{\partial x_i}
(\kappa_1(p), \ldots, \kappa_n(p)) =(-1)^{k}\sum_{i_1<\cdots<i_k,i_j\neq
i}\kappa_{i_1}(p)\cdots\kappa_{i_k}(p),
\]
for every $1\leq i\leq n$. From here it can be easily seen that
\beq
\label{trPk}
\mathrm{trace}(P_k)=c_kH_k, 
\eeq
\beq
\label{trAPk}
\mathrm{trace}(A\circ P_k)=-c_kH_{k+1}, 
\eeq
and
\beq
\label{trA2Pk}
\mathrm{trace}(A^2\circ P_k)=\bin{n}{k+1}(nH_1H_{k+1}-(n-k-1)H_{k+2}),
\eeq
where $$c_k=(n-k)\bin{n}{k}=(k+1)\bin{n}{k+1},$$
and $H_{k}=0$ if $k>n$. We refer the reader
to \cite{ABC} for further details (see also \cite{Re} and \cite{Ro} for other details about classical Newton tranformations for hypersurfaces in Riemannian spaces).

Let $\nabla$ stand for the Levi-Civita connection of \m. Associated to
each Newton transformation $P_k$, we consider the second order linear
differential operator
$L_k:\mathcal{C}^\infty(\m)\fle\mathcal{C}^\infty(\m)$ given by
\[
L_k(f)=\mathrm{trace}(P_k\circ\nabla^2f).
\]
Here $\nabla^2f:\mathcal{X}(\m)\fle\mathcal{X}(\m)$ denotes the
self-adjoint linear operator metrically equivalent to the hessian of $f$, and it
is given by
\[
\g{\nabla^2f(X)}{Y}=\g{\nabla_X(\nabla f)}{Y}, \quad X,Y\in\xm.
\]

Observe that
\begin{eqnarray*}
L_k(f) & = & \mathrm{trace}(P_k\circ\nabla^2f)=
\sum_{i=1}^n\g{P_k(\nabla_{E_i}\nabla f)}{E_i}\\ {} & = &
\sum_{i=1}^n\g{\nabla_{E_i}\nabla f}{P_k(E_i)}=
\sum_{i=1}^n\g{\nabla_{P_k(E_i)}\nabla f}{E_i}=
\mathrm{trace}(\nabla^2f\circ P_k),
\end{eqnarray*}
where $\{ E_1, \ldots, E_n \}$ is a local orthonormal frame on \m.
Moreover, we also have that
\begin{eqnarray}
\label{divPk}
\nonumber \mathrm{div}(P_k(\nabla f)) & = &
\sum_{i=1}^n\g{(\nabla_{E_i}P_k)(\nabla f)}{E_i}+\sum_{i=1}^n\g{P_k(\nabla_{E_i}\nabla f)}{E_i}\\
{} & = & \g{\mathrm{div}P_k}{\nabla f}+L_k(f),
\end{eqnarray}
where
\[
\mathrm{div}P_k:=\mathrm{trace}(\nabla
P_k)=\sum_{i=1}^n(\nabla_{E_i}P_k)(E_i).
\]

In particular, if $P_k$ is divergence-free then
$L_k(f)=\mathrm{div}(P_k(\nabla f))$ and $L_k$ is a divergence form differential operator on \m. This happens trivially when $k=0$, and actually $L_0$ is nothing but the Laplacian operator $\Delta$. On the other
hand, by \cite[Corollary 3.2]{ABC} this also happens for every $k=0, \ldots,n$
when the GRW ambient spacetime has constant sectional curvature. This is a consequence of the following useful expression for the divergence of the Newton operators, which can be found in \cite[Lemma 3.1]{ABC}.
\begin{lemma}
\label{lemadivPk}
The divergences of the Newton transformations $P_k$ of a spacelike hypersurface $\m^n$ immersed into an arbitrary $(n+1)$-dimensional spacetime are given by
\[
\g{\mathrm{div}P_k}{X}=\sum_{j=0}^{k-1}\sum_{i=1}^{n}\g{\Ro(E_i,A^{k-1-j}X)N}{P_jE_i}, \quad k=1,\ldots,n-1,
\]
for every tangent field $X\in\xm$, where $\Ro$ stands for the curvature tensor of the ambient spacetime\footnote{In our
notation and following \cite{ONe}, we use $\overline{R}(X,Y)Z=\nablao_{[X,Y]}Z-[\nablao_X,\nablao_Y]Z$}.
\end{lemma}

It follows from \rf{divPk} that the operator $L_k$ is elliptic if and only if $P_k$ is positive definite.
Clearly, $L_0=\Delta$ is always elliptic. For our applications, it will be useful to have some geometric conditions which guarantee the ellipticity of $L_k$ when $k\geq 1$. For $k=1$ the following one is esentially a consequence of Lemma 3.10 in \cite{El}.
\begin{lemma}
\label{lemafernanda}
Let \m\ be a spacelike hypersurface immersed into a GRW spacetime. If $H_2>0$ on \m, then $L_1$ is elliptic or, equivalently, $P_1$ is positive definite (for an appropriate choice of the Gauss map $N$).
\end{lemma}
\begin{proof}
Simple observe that by Cauchy-Schwarz inequality we have $H_1^2\geq H_2>0$, and $H_1$ does not vanish on \m. By choosing the appropriate Gauss map, we may assume that $H_1>0$. Recall that $H_2$ does not depend on the chosen $N$. Since $n^2H_1^2=\sum\kappa_i^2+n(n-1)H_2>\kappa_i^2$ for every $i=1,\ldots,n$, then
$\mu_{i,1}=nH_1+\kappa_i>0$ for every $i$ and $P_1$ is positive definite.
\end{proof}

On the other hand, by an elliptic point in a spacelike hypersurface we mean a point of \m\ where all the principal curvatures are negative, with respect to an appropriate choice of the Gauss map $N$. When $2\leq k\leq n-1$, the existence of an elliptic point in \m\ also implies here that the operator $L_k$ is elliptic on \m, under the assumption that $H_{k+1}$ is positive on \m\ for that choice of $N$ (if $k$ is even). Even more, we have the following result.
\begin{lemma}
\label{lemaCR}
Let \m\ be a spacelike hypersurface immersed into a GRW spacetime. If there exists an elliptic point of \m, with respect to an appropriate choice of the Gauss map $N$, and $H_{k+1}>0$ on \m, for $2\leq k\leq n-1$, then for
all $1\leq j\leq k$ the operator $L_j$ is elliptic or, equivalently, $P_j$ is positive definite (for that appropriate choice of the Gauss map, if $j$ is odd).
\end{lemma}
The proof follows from that of \cite[Proposition 3.2]{CR} (see also the proof of \cite[Proposition 3.2]{BC}), taking into account that in our case, and by our sign convention in the definition of the $j$-th mean curvatures, we have $\bin{n}{j}H_j=\sigma_j(-\kappa_1,\ldots,-\kappa_n)=(-1)^jS_j$.

\section{The operator $L_k$ acting on the height function}

The vector field given by $$K(t,x)=f(t)(\partial/\partial t)_{(t,x)}, \quad (t,x)\in\grw,$$
determines a non-vanishing future-pointing closed conformal vector field on \grw. In fact,
\beq
\label{nablaoK}
\nablao_ZK=f'(t)Z
\eeq
for every vector $Z$ tangent to \grw\ at a point $(t,x)$, where $\nablao$
denotes the Levi-Civita connection on \grw. This conformal field $K$ will
be an essential tool in our computations.

Let \x\ be a spacelike hypersurface with Gauss map $N$. The height function of \m, denoted by $h$, is the
restriction of the projection $\pi_I(t,x)=t$ to \m, that is, $h:\m\rightarrow I$ is given by $h=\pi_I\circ\psi$.
Observe that the gradient of $\pi_I$ on \grw\ is given by
\[
\nablao\pi_I=-\g{\nablao\pi_I}{\partial_t}\partial_t=-\partial_t.
\]
Then, the gradient of $h$ on \m\ is given by
\beq
\label{gradh}
\nabla h=(\nablao\pi_I)^\top=-\partial_t^\top,
\eeq
where
\[
\partial_t=\partial_t^\top-\g{N}{\partial_t}N
\]
and $\partial_t^\top\in\xm$ denotes the tangential component of
$\partial_t$. In our computations, we will also consider the function
$g(h)$, where $g:I\rightarrow\R{}$ is an arbitrary primitive of $f$. Since
$g'=f>0$, then $g(h)$ can be thought as a reparametrization of the height
function. In particular, the gradient of $g(h)$ on \m\ is
\beq
\label{gradg(h)}
\nabla g(h)=f(h)\nabla h=-f(h)\partial_t=-K^\top,
\eeq
where $K^\top$ denotes the tangential component of $K$ along the hypersurface,
\beq
\label{KT}
K=K^\top-\g{N}{K}N.
\eeq

Equation \rf{nablaoK} implies that
\beq
\label{nablaK}
\nablao_XK=f'(h)X
\eeq
for every $X\in\xm$. From here, we can easily see that
\beq
\label{nablaKT}
\nabla_XK^\top=f'(h)X-f(h)\g{N}{\partial_t}AX=f'(h)X-\g{N}{K}AX.
\eeq
Therefore, from \rf{gradg(h)} we get that
\beq
\label{hessg(h)}
\nabla _X(\nabla g(h))=-\nabla_XK^\top=-f'(h)X+\g{N}{K}AX,
\eeq
and then, by \rf{trPk} and \rf{trAPk}, we conclude that
\[
L_k(g(h))=-f'(h)\mathrm{tr}(P_k)+\g{N}{K}\mathrm{tr}(A\circ P_k)
=-c_k(f'(h)H_k+\g{N}{K}H_{k+1}).
\]

On the other hand, taking into account that
\[
\nabla h=\frac{1}{f(h)}\nabla g(h),
\]
we also get from \rf{hessg(h)} that
\beq
\label{hessh}
\nabla_X(\nabla h)=-(\log f)'(h)(X+\g{\nabla h}{X}\nabla
h)+\g{N}{\partial_t}AX,
\eeq
and therefore
\begin{eqnarray*}
L_k(h) & = & -(\log f)'(h)(\mathrm{tr}(P_k)+\g{P_k(\nabla h)}{\nabla h})
+\g{N}{\partial_t}\mathrm{tr}(A\circ P_k)\\
{} & = &-(\log f)'(h)(c_kH_k+\g{P_k(\nabla h)}{\nabla h})
-\g{N}{\partial_t}c_kH_{k+1}.
\end{eqnarray*}
For future reference, we summarize this in the following lemma.
\begin{lemma}
\label{lemma3.1}
Let \x\ be a spacelike hypersurface immersed into a GRW spacetime, with Gauss map $N$.
Let $h=\pi_I\circ\psi$ denote the height function of \m, and let $g:I\fle\R{}$ be any primitive
of $f$. Then, for every $k=0,\ldots, n-1$ we have
\beq
\label{Lk(h)}
L_k(h)=
-(\log f)'(h)(c_kH_k+\g{P_k(\nabla h)}{\nabla h})-\g{N}{\partial_t}c_kH_{k+1},
\eeq
and
\beq
\label{Lk(g)}
L_k(g(h))=
-c_k(f'(h)H_k+\g{N}{K}H_{k+1}),
\eeq
where $$c_k=(n-k)\bin{n}{k}=(k+1)\bin{n}{k+1}.$$
\end{lemma}

\section{First applications}
\label{s5}
As a first application of Lemma\rl{lemma3.1} we will prove Theorem\rl{t1} below, which extends Theorem 7 in \cite{Mo} (see also Theorem 1 in \cite{AM}) for hypersurfaces with constant mean curvature to the case of hypersurfaces for which the quotient $H_{k+1}/H_k$ is constant, where $k=0,1,\ldots,n-1$. Before stating our result, recall from \cite[Proposition 3.2 (i)]{ARS1} that if a GRW spacetime admits a compact spacelike hypersurface, then the Riemannian factor $M^n$ is necessarily compact. In that case, it is said that \grw\ is a spatially closed GRW spacetime.
\begin{theorem}
\label{t1} Let \grw\ a spatially closed GRW spacetime such that its warping function satisfies the following condition
\beq
\label{logf}
ff''-f'^2\leq 0,
\eeq
(that is, such that $-\log f$ is convex). Let $\m^n$ be a compact spacelike
hypersurface immersed into \grw\ whose $k$-th Newton transformation $P_k$
is definite on \m, for some $k=0,1\ldots, n-1$. If the quotient $H_{k+1}/H_k$ is constant, then the
hypersurface is an embedded slice $\{t_0\}\times M$, where $t_0\in I$ satisfies $f'(t_0)\neq0$ if $k\geq 1$.
\end{theorem}
As observed by Montiel in \cite{Mo}, condition \rf{logf} is a reasonably weak condition on the warping function which is sufficient in order to obtain uniqueness results for spacelike hypersurfaces in GRW spacetimes. For instance, every GRW spacetime obeying the timelike convergence condition satisfies \rf{logf}.

Before giving the proof of Theorem\rl{t1}, it is interesting to obtain some applications of it for situations where the definiteness of $P_k$ can be derived from geometric hypothesis. For instance, from Lemma\rl{lemafernanda} we have the following.
\begin{corollary}
\label{co1}
Let \grw\ a spatially closed GRW spacetime such that its warping function satisfies the condition
\[
ff''-f'^2\leq 0.
\]
Then the only compact spacelike hypersurfaces \m\ immersed into \grw\ such that $H_2>0$ on \m\ and the quotient $H_2/H_1$ is constant are the embedded slices $\{t_0\}\times M$, with $t_0\in I$ satisfying
$f'(t_0)\neq0$.
\end{corollary}
On the other hand, in order to apply  Lemma\rl{lemaCR}, it is convenient to have some geometric condition which implies the existence of an elliptic point. The following technical result is a consequence of a more general result given by the authors in \cite[Lemma 5.4 and Remark 5.7]{ABC}, jointly with Brasil Jr., and it will be essential in most of our applications. For that reason and for the sake of completeness we give here a proof of it, especially adapted to the case of hypersurfaces in GRW spacetimes.
\begin{lemma}[Existence of an elliptic point]
\label{lemaellipticpoint}
Let \x\ be a compact spacelike hypersurface immersed into a spatially closed GRW spacetime, and assume that $f'(h)$ does not vanish on \m\ (equivalently, $\psi(\m)$ is contained in a slab
\[
\Omega(t_1,t_2)=(t_1,t_2)\times M\subset\grw
\]
on which $f'$ does not vanish)
\begin{enumerate}
\item If $f'(h)>0$ on \m\ (equivalently, $f'>0$ on $(t_1,t_2)$), then there exists an elliptic point of \m\ with respect to its future-pointing Gauss map.
\item If $f'(h)<0$ on \m\ (equivalently, $f'<0$ on $(t_1,t_2)$), then there exists an elliptic point of \m\ with respect to its past-pointing Gauss map.
\end{enumerate}
\end{lemma}
\begin{proof}
When $f'(h)>0$, let us choose on \m\ the future-pointing Gauss map $N$ and let $p_0\in\m$ be the point where the height function $h$ attains its minimum on \m. Then we have $\nabla h(p_0)=0$, $\g{N}{\partial_t}(p_0)=-1$ and
by \rf{hessh} we get
\[
\nabla^2h_{p_0}(e_i,e_i)=-(\log f)'(\hmin)-\kappa_i(p_0)\geq 0
\]
for every $i=1,\ldots,n$, where $\{e_1,\ldots,e_n\}$ is a basis of principal directions at $p_0$. Since $f'(\hmin)>0$ we get from here
\[
\kappa_i(p_0)\leq -(\log f)'(\hmin)<0,
\]
as dessired. On the other hand, when $f'(h)<0$ we choose on \m\ the past-pointing Gauss map $N$ and consider $p_0\in\m$ the point where the height function $h$ attains its maximum on \m. Now we have $\nabla h(p_0)=0$, $\g{N}{\partial_t}(p_0)=1$ and by \rf{hessh} we get
\[
\nabla^2h_{p_0}(e_i,e_i)=-(\log f)'(\hmax)+\kappa_i(p_0)\leq 0
\]
for every $i=1,\ldots,n$,  with $\{e_1,\ldots,e_n\}$ a basis of principal directions at $p_0$. Since now $f'(\hmax)<0$ we get from here
\[
\kappa_i(p_0)\leq (\log f)'(\hmin)<0.
\]
\end{proof}
Now, using Lemma\rl{lemaCR} and Lemma\rl{lemaellipticpoint} we can also state the following application of Theorem\rl{t1}.
\begin{corollary}
\label{co2}
Let \grw\ a spatially closed GRW spacetime such that its warping function satisfies the condition
\[
ff''-f'^2\leq 0.
\]
Assume that $\m^n$, $n\geq 3$, is a compact spacelike hypersurface immersed into \grw\ which is contained in a slab  $\Omega(t_1,t_2)\subset\grw$ on which $f'$ does not vanish. If $H_{k+1}>0$ on \m\ for some $k\geq 2$ and one of the  the quotients $H_{j+1}/H_{j}$ is constant for some $1\leq j\leq k$, then \m\ is necessarily an embedded slice $\{t_0\}\times M$, with $t_0\in(t_1,t_2)$.
\end{corollary}

The proof of our Theorem\rl{t1} will be a consequence of the following auxiliary result, which is a generalization of Lemma 3 and Corollary 4 in \cite{AM}.
\begin{lemma}
\label{l1}
Let \grw\ a spatially closed GRW spacetime, and let $\m^n$ be a compact
spacelike hypersurface immersed into \grw. Assume that for some
$k\in\{0,1\ldots, n-1\}$, the $k$-th Newton transformation $P_k$ is
semi-definite on \m\ and the function $H_k$ does not vanish on \m. Then
the quotient $H_{k+1}/H_k$ satifies
\beq
\label{eq1}
\min\left(\frac{H_{k+1}}{H_k}\right)\leq (\log f)'(\hmax) \quad
\mathrm{and} \quad \max\left(\frac{H_{k+1}}{H_k}\right)\geq (\log
f)'(\hmin),
\eeq
where \hmin\ and \hmax\ denote, respectively, the minimum and the maximum
values of the height function on \m. In particular, if the warping
function satifies condition \rf{logf} then
\beq
\label{eq2}
\min\left(\frac{H_{k+1}}{H_k}\right)\leq (\log f)'(\hmax)\leq (\log f)'(\hmin)\leq \max\left(\frac{H_{k+1}}{H_k}\right).
\eeq
\end{lemma}
\begin{proof}
Choose on \m\ its future-pointing Gauss map $N$. Since \m\ is compact, there exist points $\pmin,\pmax\in\m$ where the height function attains its minimum and its maximum values, respectively, that is,
\[
h(\pmin)=\min_{\m}h=\hmin\leq\hmax=h(\pmax)=\max_{\m}h.
\]
In particular, $\nabla h(\pmin)=0$ and $\nabla h(\pmax)=0$, and from
\rf{gradh} we have $N(\pmin)=(\partial_t)_{\pmin}$ and $N(\pmax)=(\partial_t)_{\pmax}$.
Using this in \rf{Lk(h)}, we obtain that
\beq
\label{e4.1}
(1/c_k)L_k(h)(\pmin)=H_{k+1}(\pmin)-(\log f)'(\hmin)H_k(\pmin),
\eeq
and
\beq
\label{e4.2}
(1/c_k)L_k(h)(\pmax)=H_{k+1}(\pmax)-(\log f)'(\hmax)H_k(\pmax).
\eeq
Let us consider first the case where $P_k$ is positive semi-definite on \m. Then $H_k>0$ on \m, because
we are also assuming that $H_k\neq 0$ on \m. In that case, taking inot account that
$L_k(h)=\mathrm{tr}(P_k\circ\nabla^2h)$,
$\nabla^2h(\pmin)$ is positive semi-definite and $\nabla^2h(\pmax)$ is negative semi-definite, then
we have that $L_k(h)(\pmin)\geq 0$ and $L_k(h)(\pmax)\leq 0$. Therefore, equations \rf{e4.1} and \rf{e4.2} imply
\beq
\label{e4.3}
\max\left(\frac{H_{k+1}}{H_k}\right)\geq\frac{H_{k+1}}{H_k}(\pmin)\geq (\log f)'(\hmin),
\eeq
and
\beq
\label{e4.4}
\min\left(\frac{H_{k+1}}{H_k}\right)\leq\frac{H_{k+1}}{H_k}(\pmax)\leq (\log f)'(\hmax).
\eeq
This yields \rf{eq1}. Finally, if $-\log f$ is convex, then $(\log f)'$ is non increasing and so
\[
(\log f)'(\hmax)\leq (\log f)'(\hmin),
\]
which jointly with \rf{eq1} yields \rf{eq2}.
This completes the proof of Lemma\rl{l1} in the case where $P_k$ is positive semi-definite.
When $P_k$ is negative semi-definite, then the proof above also works, taking into account that now $H_k<0$ on \m, and $L_k(h)(\pmin)\leq 0$ and $L_k(h)(\pmax)\geq 0$.
\end{proof}

\begin{proof}[Proof of Theorem\rl{t1}]
Let us choose again the future-pointing Gauss map $N$ on \m.
Since the quotient $H_{k+1}/H_k$ is constant on \m, we know from \rf{eq2} that
$H_{k+1}/H_k=(\log f)'(\hmax)=(\log f)'(\hmin)$. Therefore
\[
(\log f)'(h)=\frac{f'(h)}{f(h)}=H_{k+1}/H_k=\alpha
\]
is constant on \m, because $(\log f)'$ is non increasing on $I$ by the hypothesis on the convexity of $-\log f$. Then, \rf{Lk(g)} reduces to
\beq
\label{Lk(g)bis}
L_k(g(h))=-c_k(f(h)+\g{N}{K})H_{k+1}.
\eeq
Observe that $H_{k+1}=\alpha H_k$ does not change sign on \m, and observe also that the function $\g{N}{K}$
satisfies
\beq
\label{e4.5}
\g{N}{K}\leq-\|K\|=-f(h)<0 \quad \mathrm{on} \quad \m.
\eeq
Hence, by \rf{Lk(g)bis} we have that $L_k(g(h))$ does not change sign on \m, which is a compact Riemannian manifold. In the case where $P_k$ is positive definite, then $L_k$ is an elliptic operator, and we may apply the classical maximum principle for such operators to conclude that the function $g(h)$ is constant on the compact \m. If $P_k$ is negative definite, then the operator $-L_k$ is now elliptic and by the same argument we also conclude that $g(h)$ is constant. But $g(t)$ being an increasing function on $t$, that means that the height function $h$ is constant on \m\ and then the hypersurface is an slice $\{t_0\}\times M$. Finally, when $k>0$ it necessarily holds $f'(t_0)\neq 0$ because slices $\{t\}\times M$ with $f'(t)=0$ are totally geodesic in \grw, and then $P_k=0$ for every $k>0$.
\end{proof}

It is worth pointing out that if the warping function satisfies the stronger condition $ff''-f'^2\leq 0$ with equality only at isolated points of $I$ at most, then Theorem\rl{t1} remains true under the assumption that $P_k$ is semi-definite and the function $H_k$ does not vanish on \m. Actually, in this case we also get from \rf{eq2} in
Lemma\rl{l1} that $H_{k+1}/H_k=(\log f)'(\hmax)=(\log f)'(\hmin)$. But the hypothesis on the warping function implies now that $(\log f)'$ is strictly decreasing on $I$. Therefore, $\hmin=\hmax$ and the height function $h$ is constant
on \m. Then, we also obtain the following related result.
\begin{theorem}
\label{t1bis} Let \grw\ a spatially closed GRW spacetime such that its warping function
satisfies the condition \rf{logf}, with equality only at isolated points of $I$ at most. Let $\m^n$ be a compact spacelike hypersurface immersed into \grw\ whose $k$-th Newton transformation $P_k$
is semi-definite on \m\ and $H_k$ does not vanish on \m, for some $k=0, 1\ldots, n-1$. If the quotient $H_{k+1}/H_k$ is constant, then the
hypersurface is an embedded slice $\{t_0\}\times M$, where $t_0$ satisfies $f'(t_0)\neq0$ if $k>0$.
\end{theorem}

\section{Minkowski formulae for hypersurfaces in GRW spacetimes}
\label{s6}
The use of Minkowski-type integral formulae for compact spacelike hypersurfaces in Lorentzian ambient spacetimes was first started by Montiel in \cite{Mo1} in his study of hypersurfaces with constant mean curvature in de Sitter space, and it was continued later by Al\'\i as, Romero and S\'{a}nchez in \cite{ARS1,ARS2,ARS3} for constant mean curvature hypersurfaces in GRW spacetimes and, more generally, in conformally stationary spacetimes. Observe that for the case of the mean curvature, only the first and second Minkowski formulae come into play. In \cite{AAR}, Aledo, Al\'\i as and Romero developed  the general Minkowski formulae for compact spacelike hypersurfaces in de Sitter space, and applied them to the study of hypersurfaces with constant $k$-th mean curvature. On the other hand, in \cite{Mo} Montiel gave also another proof of the first and second Minkowski formulae for spacelike hypersurfaces in conformally stationary spacetimes with a closed conformal field, as well as
 a proof of the third Minkowski formula for the case where the ambient spacetime has constant sectional curvature. As observed by Montiel, his method of proof, which follows the ideas of Hsiung \cite{Hs} and uses parallel hypersurfaces, has a very nice geometric interpretation but it is very difficult to carry our for succesive Minkowski formulae. The reason is that it involves covariant derivatives of the Ricci tensor of the ambient space, and one need to assume at least that the ambient space is Einstein to get some control of them. In \cite{ABC} the authors, jointly with Brasil Jr., developed another method to obtain general Minkowski formulae for spacelike hypersurfaces in conformally stationary spacetimes. Although more analytic than geometric, this method, which follows the ideas of Reilly in \cite{Re}, has the advantage of working for succesive Minkowski formulae. However, in order to get some interesting applications from them, one need to assume again that the ambien
 t space has constant sectional curvature.

In this section, and as another application of our formulae in Lemma\rl{lemma3.1}, we will derive general Minkowski-type formulae for compact spacelike hypersurfaces immersed into a GRW spacetime. The main interest of these new approach is that our new Minkowski formulae can be applied to the study of hypersurfaces with constant higher order mean curvature in arbitrary Robertson-Walker spacetimes, even when the ambient space does not have constant sectional curvature.

Let $\m^n$ be a compact spacelike hypersurface immersed into a GRW spacetime. The general Minkowski formulae for \m, as derived in \cite{ABC}, are just a simple consequence of our formula \rf{Lk(g)} and the expression \rf{divPk}. In fact, it follows from \rf{divPk} and \rf{Lk(g)} that
\[
\mathrm{div}(P_k(\nabla g(h)))=f(h)\g{\mathrm{div}P_k}{\nabla h}-c_k(f'(h)H_k+\g{N}{K}H_{k+1}),
\]
for every $k=0,\ldots, n-1$. Therefore, integrating this equality on the compact Riemannian manifold \m, the divergence theorem gives the following Minkowski formulae,
\beq
\label{MF}
c_k\int_{\Sigma}(f'(h)H_k+\g{N}{K}H_{k+1})d\m=\int_{\Sigma}f(h)\g{\mathrm{div}P_k}{\nabla h}d\m,
\eeq
where $d\m$ is the Riemannian measure induced on \m. This is just Theorem 4.1 in \cite{ABC} for the case of hypersurfaces into a GRW spacetime. In particular, if $k=0$ then $P_0=I$, $\mathrm{div}P_0=0$, and this gives the first Minkowski formula
\[
\int_{\Sigma}(f'(h)+\g{N}{K}H_{1})d\m=0,
\]
which holds true without any additional hypothesis on the ambient GRW spacetime. On the other hand, when $k=1$ we easily get from the inductive formula in Lemma\rl{lemadivPk} that
\[
\g{\mathrm{div}P_1}{\nabla h}=\sum_{i=1}^{n}\g{\overline{R}(E_i,\nabla h)N}{E_i}=-\Rico(N,\nabla h),
\]
where $\Rico$ denotes the Ricci tensor of \grw. A direct computation using the general relationship between the
Ricci tensor of a warped product and the Ricci tensors of its base and its fiber, as well as the derivatives of the warping function (see for instance \cite[Corollary 43]{ONe}), implies that, for the special case of
our GRW ambient spacetimes, we have
\begin{eqnarray}
\label{ricci}
\nonumber \Rico(U,V) & = & \mathrm{Ric}_{M}(U^*,V^*)+
(n((\log f)')^2+(\log f)'')\g{U}{V}\\
{} & {} & -(n-1)(\log f)''\g{U}{\partial_t}\g{V}{\partial_t},
\end{eqnarray}
for arbitrary vector fields $U,V$ on \grw, where, for simplicity,  we are writing $\log f=\log f(\pi_I)$,
$(\log f)'=(\log f)'(\pi_I)$, and $(\log f)''=(\log f)''(\pi_I)$. Here $\mathrm{Ric}_{M}$ denotes
the Ricci tensor of the Riemannian manifold $M^n$, and $U^*=(\pi_{M})_*(U)$ denotes projection
onto the fiber $M$ of a vector field $U$ defined on \grw, that is,
\[
U=U^*-\g{U}{\partial_t}\partial_t.
\]
Observe that the decomposition $\partial_t=-\nabla h-\g{N}{\partial_t}N$ yields
$(\nabla h)^*=-\g{N}{\partial_t}N^*$. In particular, it follows from \rf{ricci} that
\[
\g{\mathrm{div}P_1}{\nabla h}=-\Rico(N,\nabla h)=\g{N}{\partial_t}
\left(\mathrm{Ric}_{M}(N^*,N^*)-(n-1)(\log f)''(h)\|\nabla h\|^2\right).
\]
Therefore, the second Minkowski formula (when $k=1$) for a compact spacelike hypersurface into a GRW spacetime becomes
\begin{eqnarray*}
c_1\int_{\Sigma}(f'(h)H_1+\g{N}{K}H_{2})d\m= \\
\int_{\Sigma}\g{N}{K}
\left(\mathrm{Ric}_{M}(N^*,N^*)-(n-1)(\log f)''(h)\|\nabla h\|^2\right)d\m.
\end{eqnarray*}
For its further reference, we summarize here the first two Minkowski formulae for spacelike hypersurfaces in a GRW spacetime.
\begin{theorem}
Let \x\ be a compact spacelike hypersurface immersed into a GRW spacetime. Then
\beq
\label{MF1}
\int_{\Sigma}(f'(h)+\g{N}{K}H_{1})d\m=0,
\eeq
and
\begin{eqnarray}
\label{MF2}
n(n-1)\int_{\Sigma}(f'(h)H_1+\g{N}{K}H_{2})d\m= \\
\nonumber \int_{\Sigma}\g{N}{K}
\left(\mathrm{Ric}_{M}(N^*,N^*)-(n-1)(\log f)''(h)\|\nabla h\|^2\right)d\m.
\end{eqnarray}
\end{theorem}

In order to find a useful expression for the succesive Minkowski formulae given by \rf{MF}, in the case where $n\geq 3$, we will assume that \grw\ is a (classical) Robertson-Walker spacetime, that is, the Riemannian manifold $M^n$ has constant sectional curvature $\kappa$. Our objective is to compute
$\g{\mathrm{div}P_k}{\nabla h}$ when $2\leq k\leq n-1$. From Lemma\rl{lemadivPk} we know that
\beq
\label{lemadivPkbis}
\g{\mathrm{div}P_k}{\nabla h}=\sum_{j=0}^{k-1}\sum_{i=1}^{n}\g{\Ro(E_i,A^{k-1-j}(\nabla h))N}{P_jE_i},
\eeq
where $\{ E_1,\ldots,E_n\}$ is an arbitrary local orthonormal frame on \m.
A direct but heavy computation using the general relationship between the
curvature tensor of a warped product and the curvature tensors of its base
and its fiber, as well as the derivatives of the warping function (see for
instance \cite[Proposition 42]{ONe}), implies that, for our GRW ambient spacetimes, we have the following expression
\begin{eqnarray}
\label{curvatura}
\nonumber \Ro(U,V)W & = & R_{M}(U^*,V^*)W^*+
((\log f)')^2(\g{U}{W}V-\g{V}{W}U)\\
{} & {} & +(\log f)''\g{W}{\partial_t}(\g{V}{\partial_t}U-\g{U}{\partial_t}V)\\
\nonumber {} & {} & -(\log f)''(\g{V}{\partial_t}\g{U}{W}-\g{U}{\partial_t}\g{V}{W})\partial_t,
\end{eqnarray}
for arbitrary vector fields $U,V,W$ on \grw. Here $R_{M}$ denotes
the curvature tensor of the Riemannian manifold $M^n$, which in principle we do not assume to be of constant
sectional curvature. In particular, \rf{curvatura} implies
\beq
\label{6.1}
\Ro(E_i,X)N=R_{M}(E_i^*,X^*)N^*-(\log f)''(h)\g{N}{\partial_t}(\g{X}{\nabla h}E_i-\g{E_i}{\nabla h}X)
\eeq
for every tangent vector field $X\in\xm$. From the decompositions
\[
N=N^*-\g{N}{\partial_t}\partial_t, \quad
E_i=E_i^*-\g{E_i}{\partial_t}\partial_t, \quad \mathrm{and} \quad
X=X^*-\g{X}{\partial_t}\partial_t,
\]
it follows easily that
\begin{eqnarray}
\label{A1} (N^*)^\top & = & -\g{N}{\partial_t}\nabla h,\\
\label{A2} (E_i^*)^\top & = & E_i+\g{E_i}{\nabla h}\nabla h, \\
\label{A3} (X^*)^\top & = & X+\g{X}{\nabla h}\nabla h,
\end{eqnarray}
and
\begin{eqnarray}
\label{B1} \g{N^*}{N^*}_{M} & = & \frac{1}{f^2(h)}\|\nabla h\|^2,\\
\label{B2} \g{N^*}{E_i^*}_{M} & = & -\frac{1}{f^2(h)}\g{N}{\partial_t}\g{E_i}{\nabla h}, \\
\label{B3} \g{N^*}{X^*}_{M} & = & -\frac{1}{f^2(h)}\g{N}{\partial_t}\g{X}{\nabla h}.
\end{eqnarray}
Let us remark here, for its use later, that this set of equalities (\rf{A1}-\rf{B3}) holds true for every Riemannian fiber $M^n$, since we are not using yet that $M^n$ has constant sectional curvature.

If $M^n$ has constant sectional curvature $\kappa$, then we also have that
\[
(R_{M}(E_i^*,X^*)N^*)^\top=
\kappa\g{E_i^*}{N^*}_{M}(X^*)^\top-\kappa\g{X^*}{N^*}_{M}(E_i^*)^\top.
\]
Then, using equations \rf{A2} and \rf{A3} we obtain
\[
(R_{M}(E_i^*,X^*)N^*)^\top=
\frac{\kappa}{f^2(h)}\g{N}{\partial_t}\left(\g{X}{\nabla h}E_i-\g{E_i}{\nabla h}X\right),
\]
which jointly with \rf{6.1} yields
\[
(\Ro(E_i,X)N)^\top=\left(\frac{\kappa}{f^2(h)}-(\log f)''(h)\right)\g{N}{\partial_t}
\left(\g{X}{\nabla h}E_i-\g{E_i}{\nabla h}X\right).
\]
Therefore, for every fixed $j=0,\ldots,k-1$ we find
\begin{eqnarray*}
\sum_{i=1}^{n}\g{\Ro(E_i,X)N}{P_jE_i}= \\
\left(\frac{\kappa}{f^2(h)}-(\log f)''(h)\right)\g{N}{\partial_t}
\left(\mathrm{tr}(P_j)\g{X}{\nabla h}-\g{P_jX}{\nabla h}\right)
\end{eqnarray*}
for every tangent vector field $X\in\xm$. Using this expression into \rf{lemadivPkbis}, we get
\begin{eqnarray}
\label{lemadivPkbiss}
\nonumber \g{\mathrm{div}P_k}{\nabla h} & = &
\left(\frac{\kappa}{f^2(h)}-(\log f)''(h)\right)\g{N}{\partial_t} \\
{} & {} &
\sum_{j=0}^{k-1}
\left(\mathrm{tr}(P_j)\g{A^{k-1-j}\nabla h}{\nabla h}-\g{(P_j\circ A^{k-1-j})\nabla h}{\nabla h}\right).
\end{eqnarray}
Now we claim that, in general,
\beq
\label{suma}
\sum_{j=0}^{k-1}
\left(\mathrm{tr}(P_j)A^{k-1-j}-P_j\circ A^{k-1-j}\right)=(n-k)P_{k-1},
\eeq
for every $2\leq k\leq n$. We will prove \rf{suma} by induction on $k$. When $k=2$, using \rf{trPk} we easily see that \rf{suma} reduces to
\[
(n-1)A+n(n-1)H_1I-P_1=(n-1)P_1-P_1=(n-2)P_1.
\]
Now assume that \rf{suma} holds true for $k-1\geq 2$, and write
\begin{eqnarray*}
\sum_{j=0}^{k-1}
\left(\mathrm{tr}(P_j)A^{k-1-j}-P_j\circ A^{k-1-j}\right) & = &
\sum_{j=0}^{k-2}
\left(\mathrm{tr}(P_j)A^{k-2-j}-P_j\circ A^{k-2-j}\right)\circ A \\
{} & {} & +\mathrm{tr}(P_{k-1})I-P_{k-1}.
\end{eqnarray*}
Using the induction hypothesis here we get that
\begin{eqnarray*}
\sum_{j=0}^{k-1}
\left(\mathrm{tr}(P_j)A^{k-1-j}-P_j\circ A^{k-1-j}\right) & = &
(n-k+1)P_{k-2}\circ A\\
{} & {} & +(n-k+1)\bin{n}{k-1}H_{k-1}I-P_{k-1}\\
{} & = & (n-k+1)P_{k-1}-P_{k-1}\\
{} & = & (n-k)P_{k-1},
\end{eqnarray*}
as dessired. Finally, using \rf{suma} into \rf{lemadivPkbiss} we conclude that
\beq
\label{divPKfinal}
\g{\mathrm{div}P_k}{\nabla h}=(n-k)\left(\frac{\kappa}{f^2(h)}-(\log f)''(h)\right)\g{N}{\partial_t}
\g{P_{k-1}\nabla h}{\nabla h}
\eeq
for every $k\geq 2$. Putting this expression into \rf{MF}, we obtain the following version of higher order Minkowski formulae for spacelike hypersurfaces in RW spacetimes.
\begin{theorem}
\label{thMF}
Let \x\ be a compact spacelike hypersurface immersed into a RW spacetime, $M^n$ being a Riemannian manifold with constant sectional curvature $\kappa$. Then, for every $k=2,\ldots,n-1$
\begin{eqnarray}
\label{MFk}
\nonumber \bin{n}{k}\int_{\Sigma}(f'(h)H_k+\g{N}{K}H_{k+1})d\m=\\
\int_{\Sigma}
\left(\frac{\kappa}{f^2(h)}-(\log f)''(h)\right)\g{N}{K}\g{P_{k-1}\nabla h}{\nabla h}d\m.
\end{eqnarray}
\end{theorem}
Formulas \rf{MF1} and \rf{MF2} are refered to as the first and second Minkowski formulae, respectively.
In the following, we will refer to formula \rf{MFk} as the $(k+1)$-th Minkowski formula, with $k=2,\ldots,n-1$.

\section{Umbilicity of hypersurfaces in RW spacetimes}
\label{s7}
In \cite[Theorem 8]{Mo} Montiel gave a uniqueness result for compact spacelike hypersurfaces with constant scalar
curvature immersed into a constant sectional curvature spacetime which is equipped with a closed conformal timelike
vector field (see also \cite[Theorem 5.3]{ABC} for the case where the conformal timelike
vector field is not necessarily closed). As observed by Montiel in \cite{Mo}, every spacetime having a closed conformal
timelike vector field is locally isometric to a GRW spacetime. On the other hand, it is not difficult to see that a GRW
spacetime \grw\ has constant sectional curvature $\bar{\kappa}$ if and only if the Riemannian $M^n$ has constant
sectional curvature $\kappa$ (that is, \grw\ is in fact a RW spacetime) and the warping function $f$ satisfies the
following differential equations
\[
\frac{f''}{f}=\bar{\kappa}=\frac{(f')^2+\kappa}{f^2}
\]
(see, for instance, \cite[Corollary 9.107]{Be}). Therefore, Montiel's result \cite[Theorem 8]{Mo} esentially states as follows.
\begin{theorem}
\label{montielTh8}
Let \grw\ be a spatially closed RW spacetime with constant sectional curvature $\bar{\kappa}$, with $n\geq 3$.
Then every compact spacelike hypersurface \m\ immersed into \grw\ with constant scalar curvature $S$ such that
$S<n(n-1)\bar{\kappa}$ is totally umbilical. Moreover, if $\bar{\kappa}\leq 0$ then \m\ is an embedded slice
$\{t\}\times M^n$ (necessarily with $f'(t)\neq 0$), and if $\bar{\kappa}>0$ then \m\ is either an embedded slice
$\{t\}\times M^n$ (necessarily with $f'(t)\neq 0$) or a round umbilical hypersphere in a de Sitter space.
\end{theorem}

As a first application of our Minkowski formulae in Theorem\rl{thMF}, we extend this result to the case of hypersurfaces
immersed into a RW spacetime obeying the null convergence condition (NCC). Recall that a spacetime obeys NCC if its
Ricci curvature is non-negative on null (or lightlike) directions. Obviously, every spacetime with constant sectional
curvature trivially obeys NCC. Using expression \rf{ricci}, it is not difficult to see that a GRW spacetime \grw\ obeys
NCC if and only if
\beq
\label{NCC}
\mathrm{Ric}_{M}\geq (n-1)\sup_I(ff''-f'^2)\g{}{}_{M},
\eeq
where $\mathrm{Ric}_{M}$ and $\g{}{}_{M}$ are respectively the Ricci and metric tensors of the Riemannian manifold $M^n$.
In particular, a RW spacetime obeys NCC if and only if
\beq
\label{NCCbis}
\kappa\geq\sup_I(ff''-f'^2),
\eeq
where $\kappa$ denotes the constant sectional curvature of $M^n$. On the other hand, the condition that the scalar
curvature $S$ of \m\ is constant and less that $n(n-1)\bar{\kappa}$ in Theorem\rl{montielTh8} is equivalent, from
equation \rf{scalar}, to the fact that $H_2$ is a positive constant. Then, Theorem\rl{montielTh8} admits the following
extension.
\begin{theorem}
\label{th}
Let \grw\ be a spatially closed RW spacetime obeying the null convergence condition, with $n\geq 3$.
Then every compact spacelike hypersurface immersed into \grw\ with positive constant $H_2$ is totally umbilical. Moreover,
\m\ must be a slice $\{t_0\}\times M^n$ (necessarily with $f'(t_0)\neq 0$), unless in the case where \grw\ has positive
constant sectional curvature and \m\ is a round umbilical hypersphere.
The latter case cannot occur if the inequality in \rf{NCCbis} is strict.
\end{theorem}
\begin{proof}
Multiplying the first Minkowski formula \rf{MF1} by the constant $\binom{n}{2}H_2$, and substracting the third Minkowski
formula (formula \rf{MFk} with $k=2$), we obtain
\beq
\label{here}
\int_{\Sigma}
\left(\bin{n}{2}(H_1H_{2}-H_3)+\left(\frac{\kappa}{f^2(h)}-(\log f)''(h)\right)\g{P_{1}\nabla h}{\nabla h}\right)\g{N}{K}d\m=0.
\eeq
Since $H_2>0$ we know from Lemma\rl{lemafernanda} that, for an appropriate choice of the Gauss map $N$, the
transformation $P_1$ is positive definite and $H_1>0$ on \m. Since we always have $H^2_2-H_1H_3\geq 0$ (see, for instance,
\cite[Theorem 51 and Theorem 144]{HLP}), then in our case we get
\[
H_1H_{2}-H_3\geq\frac{H_2}{H_1}(H_1^2-H_2)\geq 0,
\]
by Cauchy-Schwarz inequality, with equality if and only if \m\ is totally umbilical. On the other hand, since $P_1$ is
positive definite, we also have by NCC \rf{NCCbis} that
\[
\left(\frac{\kappa}{f^2(h)}-(\log f)''(h)\right)\g{P_{1}\nabla h}{\nabla h}\geq 0.
\]
Therefore, since $\g{N}{K}$ does not vanish on \m, we conclude from \rf{here} that \m\ is totally umbilical and that
\[
\left(\frac{\kappa}{f^2(h)}-(\log f)''(h)\right)\g{P_{1}\nabla h}{\nabla h}\equiv 0 \quad \mathrm{on} \quad \m.
\]
Since $P_1$ is positive definite, it follows from here that either \m\ is a slice (equivalently, $\nabla h\equiv 0$) or
\beq
\label{k1}
\frac{\kappa}{f^2(h)}-(\log f)''(h)\equiv 0 \quad \mathrm{on} \quad \m.
\eeq

First of all, observe that if \m\ is a slice $\{t_0\}\times M^n$, then necessarily $f'(t_0)\neq 0$ because of the condition
$H_2>0$ (recall that slices with $f'(t_0)=0$ are totally geodesic).
On the other hand, if \m\ is not a slice, since it we already know that it is totally umbilical we have that $H_2=H_1^2$, which implies that \m\ has also constant mean curvature. Therefore, applying the classification of
compact umbilical spacelike hypersurfaces with constant mean curvature given by Montiel in \cite[Theorem 5]{Mo}, we
conclude that the only case in which \m\ is not a slice occurs only when \grw\ has positive
constant sectional curvature and \m\ is a round umbilical hypersphere.
Finally, observe that if the inequality in \rf{NCCbis} is strict, then \rf{k1} cannot occur and \m\ is necessarily a slice.
\end{proof}

For the general case of hypersurfaces with constant higher order mean curvature, we can state the following version
of Theorem\rl{th}.
\begin{theorem}
\label{thbis}
Let \grw\ be a spatially closed RW spacetime obeying the null convergence condition, with $n\geq 3$.
Assume that $\m^n$ is a compact spacelike hypersurface immersed into \grw\ which is contained in a slab
$\Omega(t_1,t_2)\subset\grw$ on which
$f'$ does not vanish. If $H_{k}$ is constant, with $3\leq k\leq n$ then \m\ is totally umbilical. Moreover,
\m\ must be a slice $\{t_0\}\times M^n$ (necessarily with $f'(t_0)\neq 0$), unless in the case where \grw\ has positive
constant sectional curvature and \m\ is a round umbilical hypersphere.
The latter case cannot occur if the inequality in \rf{NCCbis} is strict.
\end{theorem}
\begin{proof}
The proof follows the ideas of that of Theorem 5.10 in \cite{ABC}. Let us assume,
for instance, that $f'>0$ on $(t_1,t_2)$. From Lemma\rl{lemaellipticpoint}, we know
that, with respect to the future-pointing Gauss map $N$, there exists a point
$p_0\in\m$ where all the principal curvatures are negative. Therefore, the constant
$H_k=H_k(p_0)$ is positive, and following the ideas of Montiel and Ros in
\cite[Lemma 1]{MR} and their use of Garding inequalities \cite{Ga} (taking into
account our sign convention in the definition of the $j$-th mean curvatures) we get
that \beq \label{garding} H_1\geq H_2^{1/2}\geq\cdots\geq H_{k-1}^{1/(k-1)}\geq
H_{k}^{1/k}>0 \quad \mathrm{on} \quad \m. \eeq Moreover, equality at any stage
happens only at umbilical points. Therefore, since $f'(h)>0$ on \m, we have that
$f'(h)H_{k-1}\geq f'(h)H_{k}^{(k-1)/k}$. Integrating this inequality, and using the
first \rf{MF1} and the $k$-th Minkowski formulae (that is, \rf{MFk} with $k-1$
instead of $k$; observe that $2\leq k-1\leq n-1$), we get
\begin{eqnarray*}
H_k\int_{\Sigma}\g{N}{K}d\m=-\int_{\Sigma}f'(h)H_{k-1}d\m\leq \\
-H_{k}^{(k-1)/k}\int_{\Sigma}f'(h)d\m=
H_k^{(k-1)/k}\int_{\Sigma}\g{N}{K}H_1d\m.
\end{eqnarray*}
That is,
\beq
\label{INE}
\int_{\Sigma}(H_1-H_k^{1/k})\g{N}{K}d\m\geq 0.
\eeq
Since $N$ is future-pointing, we know that $\g{N}{K}\leq-f(h)<0$, and by \rf{garding} we have $H_1-H_k^{1/k}\geq 0$. Therefore, \rf{INE} implies that $H_1-H_k^{1/k}\equiv 0$ on \m, and hence the hypersurface is totally umbilical. Since $H_k$ is a positive constant and \m\ is totally umbilical, we know that all the higher order mean curvatures are constant. In particular, $H_2$ is a positive constant and the result follows by applying Theorem\rl{th}.
\end{proof}

When $n\geq 4$ and $H_k$ is constant with $k\leq n-1$, there is still another version of Theorem\rl{thbis} under the
assumption of existence of an elliptic point.
\begin{theorem}
\label{thbiss}
Let \grw\ be a spatially closed RW spacetime obeying the null convergence condition, with $n\geq 4$.
Assume that $\m^n$ is a compact spacelike hypersurface immersed into \grw\ which contains an elliptic point. If $H_{k}$ is constant, with $3\leq k\leq n-1$ then \m\ is totally umbilical. Moreover,
\m\ must be a slice $\{t_0\}\times M^n$ (necessarily with $f'(t_0)\neq 0$), unless in the case where \grw\ has positive
constant sectional curvature and \m\ is a round umbilical hypersphere. The latter case cannot occur if the inequality in \rf{NCCbis} is strict.
\end{theorem}
\begin{proof}
The proof follows the ideas of that of Theorem\rl{th}. Multiplying now the first Minkowski formula \rf{MF1} by the
constant $\binom{n}{k}H_{k}$, and substracting the $(k+1)$-th Minkowski formula (formula \rf{MFk}; recall that
$3\leq k\leq n-1$)), we obtain
\beq
\label{here2}
\int_{\Sigma}
\left(\bin{n}{k}(H_1H_{k}-H_{k+1})+\left(\frac{\kappa}{f^2(h)}-(\log f)''(h)\right)\g{P_{k-1}\nabla h}{\nabla h}\right)\g{N}{K}d\m=0.
\eeq
Assume that $p_0\in\m$ is an elliptic point, that is, a point where all the principal curvatures are negative. Therefore,
the constant $H_k=H_k(p_0)$ is positive, and
following the ideas of Montiel and Ros in \cite[Lemma 1]{MR} and their use of Garding inequalities \cite{Ga}, as in the proof of Theorem\rl{thbis}, we get \rf{garding}, with equality at any stage only at umbilical points. On the other hand, for every $1\leq j\leq n$ one has
the following generalized Cauchy-Schwarz inequalities (see, for instance,
\cite[Theorem 51 and Theorem 144]{HLP}),
\[
H_{j-1}H_{j+1}\leq H_{j}^2.
\]
Since each $H_j>0$ for $j=1,\ldots,k$ by \rf{garding}, this is equivalent to
\[
\frac{H_{k+1}}{H_{k}}\leq\frac{H_{k}}{H_{k-1}}\leq\cdots\frac{H_{2}}{H_{1}}\leq H_1.
\]
But this implies
\beq
\label{garding2}
H_1H_k-H_{k+1}\geq 0,
\eeq
with equality if and only if \m\ is totally umbilical.

On the other hand, by Lemma\rl{lemaCR} we also know that the transformation $P_{k-1}$ is positive definite on \m, which
jointly with \rf{NCCbis} yields
\[
\left(\frac{\kappa}{f^2(h)}-(\log f)''(h)\right)\g{P_{k-1}\nabla h}{\nabla h}\geq 0.
\]
Therefore, since $\g{N}{K}$ does not vanish on \m, we conclude from \rf{here2} that \m\ is totally umbilical and that
\[
\left(\frac{\kappa}{f^2(h)}-(\log f)''(h)\right)\g{P_{k-1}\nabla h}{\nabla h}\equiv 0 \quad \mathrm{on} \quad \m.
\]
The proof then finishes as that of Theorem\rl{th}.
\end{proof}

\section{The operator $L_k$ acting on the function $\g{N}{K}$}
\label{s8}
In this section we will compute the operator $L_k$ acting on the function $\g{N}{K}$. From \rf{nablaK} we see that
\[
X(\g{N}{K})=-\g{A(X)}{K}=-\g{X}{A(K^\top)}
\]
for every vector field $X\in\xm$, so that
\[
\nabla\g{N}{K}=-A(K^\top).
\]
Therefore,
\beq
\label{hessNK1}
\nabla_X(\nabla\g{N}{K})=-\nabla_X(A(K^\top))=-(\nabla_XA)(K^\top))-A(\nabla_XK^\top),
\eeq
where $\nabla_XA$ denotes the covariant derivative of $A$,
\[
(\nabla_XA)(Y)=\nabla A(Y,X)=\nabla_X(AY)-A(\nabla_XY), \quad X,Y\in\xm.
\]
Recall now that the Codazzi equation of a hypersurface describes the normal
component of $\overline{R}(X,Y)Z$, for tangent vector fields $X,Y,Z\in\xm$, in terms of the derivative of the shape
operator. Specifically, the Codazzi equation of a spacelike hypersurface immersed into an arbitrary ambient spacetime is
given by
\[
\g{\overline{R}(X,Y)Z}{N}=\g{(\nabla_YA)X-(\nabla_XA)Y}{Z}
\]
or, equivalently,
\beq
\label{codazzi2}
(\overline{R}(X,Y)N)^\top=(\nabla_XA)Y-(\nabla_YA)X.
\eeq
Then, using  \rf{nablaKT} and \rf{codazzi2} into \rf{hessNK1}, we obtain that
\[
\nabla_X(\nabla\g{N}{K})=
-(\nabla_{K^\top}A)(X))-(\Ro(X,K^\top)N)^\top-f'(h)A(X)+\g{N}{K}A^2(X).
\]
This implies
\begin{eqnarray}
\label{LkNK1}
\nonumber L_k(\g{N}{K}) & = &
-\mathrm{tr}((\nabla_{K^\top}A)\circ P_{k})
-\sum_{i=1}^{n}\g{\overline{R}(E_i,K^\top)N}{P_k(E_i)} \\
\nonumber {} & {} &
-f'(h)\mathrm{tr}(A\circ P_k)+\g{N}{K}\mathrm{tr}(A^2\circ P_k)\\
{} & = &
\bin{n}{k+1}\g{\nabla H_{k+1}}{K}-
\sum_{i=1}^{n}\g{\overline{R}(E_i,K^\top)N}{P_k(E_i)} \\
\nonumber {} & {} &
+f'(h)c_kH_{k+1}+\bin{n}{k+1}\g{N}{K}(nH_1H_{k+1}-(n-k-1)H_{k+2}),
\end{eqnarray}
where $\{ E_1, \ldots, E_n \}$ is an arbitrary local orthonormal frame on \m. Here we are using the general fact that
\[
\mathrm{tr}(\nabla_XA\circ P_k)=\mathrm{tr}(P_k\circ \nabla_XA)=
-\bin{n}{k+1}\g{\nabla H_{k+1}}{X}
\]
for every tangent vector field $X\in\xm$ (see equation (3.6) in \cite{ABC}).

In order to write equation \rf{LkNK1} in a more suitable way, observe that from \rf{KT} we can
express
\beq
\label{Ro1}
\Ro(X,K^\top)N=\Ro(X,K)N-\g{N}{K}\Ro(N,X)N 
\eeq
for every tangent vector field $X\in\xm$. On the other hand, since $K^*=0$, it follows from \rf{curvatura} that
\[
\Ro(X,K)N=-\frac{f''(h)}{f(h)}\g{N}{K}X
\]
for every tangent vector field $X\in\xm$, so that \rf{Ro1} becomes
\[
\Ro(X,K^\top)N=-\g{N}{K}\left(\frac{f''(h)}{f(h)}X+\Ro(N,X)N\right).
\]
Therefore,
\begin{eqnarray}
\label{A}
\nonumber \sum_{i=1}^{n}\g{\overline{R}(E_i,K^\top)N}{P_kE_i}=
-\g{N}{K}\left(\frac{f''(h)}{f(h)}\mathrm{tr}(P_k)+\mathrm{tr}(P_k\circ\Ro_N)\right)\\
=-\g{N}{K}\left(\frac{f''(h)}{f(h)}c_kH_k+\mathrm{tr}(P_k\circ\Ro_N)\right),
\end{eqnarray}
where $\Ro_N:\xm\rightarrow\xm$ is the operator given by $\Ro_N(X)=(\Ro(N,X)N)^\top$.
Using \rf{A} in \rf{LkNK1} we obtain the following result.
\begin{lemma}
\label{LkNK}
Let $\m^n$ be a spacelike hypersurface immersed into a GRW spacetime \grw, with Gauss map $N$ and height function $h$. Then,
for every $k=0,\ldots, n-1$ we have
\begin{eqnarray}
\label{Lk(NK)1}
L_{k}(\g{N}{K}) & = & \bin{n}{k+1}\g{\nabla H_{k+1}}{K}+f'(h)c_{k}H_{k+1}\\
\nonumber {} & {} & +\bin{n}{k+1}\g{N}{K}(nH_1H_{k+1}-(n-k-1)H_{k+2}) \\
\nonumber {} & {} & +\g{N}{K}\left(\mathrm{tr}(P_{k}\circ\Ro_N)+\frac{f''(h)}{f(h)}c_{k}H_{k}\right),
\end{eqnarray}
where $\Ro_N:\xm\rightarrow\xm$ is the operator given by $\Ro_N(X)=(\Ro(N,X)N)^\top$.
\end{lemma}

Formula \rf{Lk(NK)1} becomes specially simple when $k=0$. In that case
$L_0=\Delta$ is the Laplacian operator on \m, and
\beq
\label{ricciN}
\mathrm{tr}(P_0\circ\Ro_N)=\mathrm{tr}(\Ro_N)=\Rico(N,N),
\eeq
where $\Rico$ is the Ricci tensor of \grw. Therefore, taking into account that
\beq
\label{Nt^2}
\g{N}{\partial_t}^2=1+\|\nabla h\|^2,
\eeq
we obtain from \rf{ricciN} and \rf{ricci} that
\beq
\label{C1}
\mathrm{tr}(\Ro_N)+n\frac{f''(h)}{f(h)}=\mathrm{Ric}_{M}(N^*,N^*)
-(n-1)(\log f)''(h)\|\nabla h\|^2.
\eeq
Thus, when $k=0$ expression \rf{Lk(NK)1} gives the following.
\begin{corollary}
Let $\m^n$ be a spacelike hypersurface immersed into a GRW spacetime, with Gauss map $N$ and height function $h$. Then
\begin{eqnarray}
\label{laplaciano}
\nonumber \Delta\g{N}{K} & = &
n\g{\nabla H}{K}+nf'(h)H+\g{N}{K}\|A\|^2 \\
{} & {} & +\g{N}{K}\left(\mathrm{Ric}_{M}(N^*,N^*)-(n-1)(\log f)''(h)\|\nabla h\|^2\right)
\end{eqnarray}
\end{corollary}

For a proof, simply use \rf{C1} in \rf{Lk(NK)1} with $k=0$ and observe that
$n^2H^2-n(n-1)H_{2}=\mathrm{tr}(A^2)=\|A\|^2$.

In order to find a better expression of \rf{Lk(NK)1} when $k\geq 1$,
observe that from \rf{curvatura} we also get that
\begin{eqnarray*}
\Ro(N,X)N & =  & R_{M}(N^*,X^*)N^*-
(((\log f)')^2(h)+(\log f)''(h)\g{N}{\partial_t}^2)X\\
{} & {} & +(\log f)''(h)\g{X}{\partial_t}\partial_t+
(\log f)''(h)\g{N}{\partial_t}\g{X}{\partial_t}N,
\end{eqnarray*}
for every tangent vector field $X\in\xm$. Using $\nabla h=-\partial_t^\top$ and \rf{Nt^2}, this implies
\begin{eqnarray*}
\Ro_N(X) & = & (R_{M}(N^*,X^*)N^*)^\top+(\log f)''(h)\g{X}{\nabla h}\nabla h\\
{} & {} & -(\frac{f''(h)}{f(h)}+(\log f)''(h)\|\nabla h\|^2)X.
\end{eqnarray*}
It follows from here that
\begin{eqnarray}
\label{C2}
\nonumber \mathrm{tr}(P_k\circ\Ro_N)+\frac{f''(h)}{f(h)}c_kH_k & = &
\sum_{i=1}^{n}\g{R_{M}(N^*,E_i^*)N^*}{P_kE_i}\\
{} & {} & -(\log f)''(h)\left(c_kH_k\|\nabla h\|^2-\g{P_k\nabla h}{\nabla h}\right),
\end{eqnarray}
where $\{ E_1, \ldots, E_n \}$ is an arbitrary local orthonormal frame on \m.

Let us consider first the case where either $n=2$ or $n\geq 3$ and $M^n$ has constant
sectional curvature. In both cases, we have that
\[
(R_{M}(N^*,X^*)N^*)^\top=
\kappa\g{N^*}{N^*}_{M}(X^*)^\top-\kappa\g{N^*}{X^*}_{M}(N^*)^\top,
\]
where $\kappa$ is either the (non-necessarily constant) Gaussian curvature of $M^2$ along
the immersion $\psi$ (when $n=2$) or the constant sectional curvature of
$M^n$ (when $n\geq 3$). Using now the equalities \rf{A1}, \rf{A3}, \rf{B1} and \rf{B3}, we obtain that
\[
(R_{M}(N^*,X^*)N^*)^\top=
\frac{\kappa}{f^2(h)}\left(\|\nabla h\|^2X-\g{X}{\nabla h}\nabla h\right),
\]
and \rf{C2} becomes
\beq
\label{C3}
\mathrm{tr}(P_k\circ\Ro_N)+\frac{f''(h)}{f(h)}c_kH_k=
\left(\frac{\kappa}{f^2(h)}-(\log f)''(h)\right)
\left(c_kH_k\|\nabla h\|^2-\g{P_k\nabla h}{\nabla h}\right).
\eeq
Using this expression into \rf{Lk(NK)1}, we obtain the following consequences.
\begin{corollary}
Let \xdos\ be a spacelike surface immersed into a 3-dimensional GRW spacetime, with Gauss map $N$ and height function $h$. Then,
\begin{eqnarray*}
L_1(\g{N}{K}) & = & \g{\nabla H_{2}}{K}+2f'(h)H_{2}+\g{N}{K}2H_1H_{2} \\
\nonumber {} & {} & -\g{N}{K}\left(\frac{K_{M}(\pi)}{f^2(h)}-(\log f)''(h)\right)
\g{A\nabla h}{\nabla h},
\end{eqnarray*}
where $\pi:\m^2\rightarrow M^2$ denotes the projection of $\m^2$ onto $M^2$,
$\pi=\pi_{M}\circ\psi$.
\end{corollary}
\begin{corollary}
Let $\m^n$ be a spacelike hypersurface immersed into a RW spacetime with Riemannian fiber $M^n$ of constant sectional
curvature $\kappa$, and let $N$ and $h$ denote its Gauss map and height function, respectively. Then,
for every $k=0,\ldots,n-1$ we have
\begin{eqnarray*}
L_k(\g{N}{K}) & = & \bin{n}{k+1}\g{\nabla H_{k+1}}{K}+f'(h)c_kH_{k+1}\\
\nonumber {} & {} & +\bin{n}{k+1}\g{N}{K}(nH_1H_{k+1}-(n-k-1)H_{k+2}) \\
\nonumber {} & {} &
+\g{N}{K}\left(\frac{\kappa}{f^2(h)}-(\log f)''(h)\right)
\left(c_kH_k\|\nabla h\|^2-\g{P_k\nabla h}{\nabla h}\right).
\end{eqnarray*}
\end{corollary}

To proceed on and obtain a better expression of \rf{Lk(NK)1} in the most general case, let us analyze the summatory term in \rf{C2}. We will compute it by using a local orthonormal frame on \m\ that
diagonalizes $A$. It is worth pointing out that such a frame does not always exist; problems
occur when the multiplicity of the principal curvatures changes. For that
reason, we will work on the subset $\m_0$ of \m\ consisting of points at which the number
of distinct principal curvatures is locally constant, which is an open dense subset of
\m\ \cite[Paragraph 16.10]{Be}. Then, for every $p\in\m_0$ there exists a local orthonormal frame defined on a
neighbourhood of $p$ that diagonalizes $A$ and, hence, diagonalize $P_k$, that is,
$\{E_1, \ldots,E_n \}$ such that $A(E_i)=\kappa_iE_i$ and
$P_kE_i=\mu_{i,k}E_i$. Therefore, denoting by $K_{M}(N^*\wedge E_i^*)$ the sectional curvature
in $M^n$ of the 2-plane generated by $N^*$ and $E_i^*$, we find that
\[
\g{R_{M}(N^*,E_i^*)N^*}{P_kE_i}=
\frac{\mu_{i,k}}{f^2(h)}
K_{M}(N^*\wedge E_i^*)\|N^*\wedge E_i^*\|^2,
\]
where we are using the fact that
\[
\|N^*\wedge E_i^*\|^2=f^4(h)\|N^*\wedge E_i^*\|^2_{M}.
\]
Then we get the following.
\begin{corollary}
Let $\m^n$ be a spacelike hypersurface immersed into a GRW spacetime, with Gauss map $N$ and height function $h$. Then, for every $k=0,\ldots, n-1$ we have
\begin{eqnarray}
\label{Lk(NK)2}
L_k(\g{N}{K}) & = & \bin{n}{k+1}\g{\nabla H_{k+1}}{K}+f'(h)c_kH_{k+1}\\
\nonumber {} & {} & +\bin{n}{k+1}\g{N}{K}(nH_1H_{k+1}-(n-k-1)H_{k+2}) \\
\nonumber {} & {} & +
\frac{\g{N}{K}}{f^2(h)}\sum_{i=1}^{n}\mu_{i,k}K_{M}(N^*\wedge E_i^*)\|N^*\wedge E_i^*\|^2 \\
\nonumber {} & {} & -\g{N}{K}(\log f)''(h)\left(c_kH_k\|\nabla h\|^2-\g{P_k\nabla h}{\nabla h}\right).
\end{eqnarray}
\end{corollary}


\section{Umbilicity of hypersurfaces in GRW spacetimes}
\label{s9}
When $k=0$, our formulas for the operator $L_0=\Delta$ acting on the functions $g(h)$ and $\g{N}{K}$ allow us to re-obtain the uniqueness result given by Montiel in \cite[Theorem 6]{Mo} for hypersurfaces with constant mean curvature.
\begin{theorem}
\label{th22}
Let \grw\ be a spatially closed GRW spacetime obeying the null convergence condition, that is,
satisfying
\beq
\label{NCCBIS}
\mathrm{Ric}_{M}\geq (n-1)\sup_I(ff''-f'^2)\g{}{}_{M},
\eeq
where $\mathrm{Ric}_{M}$ and $\g{}{}_{M}$ are respectively the Ricci and metric tensors of the compact Riemannian manifold
$M^n$.
Then the only compact spacelike hypersurfaces immersed into \grw\ with constant mean curvature are the embedded slices $\{t\}\times M^n$, $t\in I$, unless in the case where \grw\ is isometric to de Sitter spacetime in a neighborhood of \m, which must be a round umbilical hypersphere. The latter case cannot occur if we assume that the inequality in \rf{NCCBIS} is strict.
\end{theorem}
\begin{proof}
Let us choose on \m\ the future-pointing Gauss map $N$ and let us consider the function $\phi=Hg(h)+\g{N}{K}\in\mathcal{C}^\infty(\m)$. Since $H$ is constant, from \rf{Lk(g)} and \rf{laplaciano} we have that
\beq
\label{laplaciano2}
\Delta\phi=\g{N}{K}\left(\|A\|^2-nH^2+\mathrm{Ric}_{M}(N^*,N^*)-(n-1)(\log f)''(h)\|\nabla h\|^2\right).
\eeq
Observe that, by Cauchy-Schwarz inequality, $\|A\|^2-nH^2\geq 0$ on \m, with equality at umbilical points. Moreover,
from \rf{NCCBIS} and \rf{B1} we also get
\[
\mathrm{Ric}_{M}(N^*,N^*)-(n-1)(\log f)''(h)\|\nabla h\|^2\geq 0.
\]
Therefore, since $\g{N}{K}<0$ on \m, by \rf{laplaciano2} we get that $\phi$ is superharmonic on \m. But \m\ being compact we have that $\phi$ must be constant. Hence $\Delta\phi=0$, and $\|A\|^2-nH^2$ vanishes on \m, which means
that \m\ is totally umbilical, and
\beq
\label{equality}
\mathrm{Ric}_{M}(N^*,N^*)-(n-1)(\log f)''(h)\|\nabla h\|^2=0 \quad \mathrm{on} \quad \m.
\eeq
If the inequality in \rf{NCCBIS} is strict, \rf{equality} is equivalent to $N^*(p)=0$ at every $p\in\m$, that is, $\|\nabla h(p)\|=0$ at every $p\in\m$. Then, $h$ must be constant and  \m\ is a slice. In the general case, we obtain that \m\ is totally umbilical and has constant mean curvature in \grw. Then, the result follows by observing
that the case where the hypersurface is totally umbilical, but not a slice, can only occur if \grw\ is locally
a de Sitter spacetime and \m\ is a round umbilical hypersphere (see the classification of compact spacelike hypersurfaces with constant mean curvature given by Montiel in \cite[Theorem 5]{Mo}).
\end{proof}

Our objective now is to extend the reasoning above to the case of hypersurfaces with constant higher order mean
curvature, using our general formulas for the operators $L_k$ acting on the functions $g(h)$ and $\g{N}{K}$. Specifically, we will prove the following uniqueness result, which extends Theorem\rl{thbis} to the case of GRW ambient spacetimes. Here, instead of the null convergence condition we need to impose on \grw\ the following stronger condition
\beq
\label{NCCstrong}
K_M\geq\sup_I(ff''-f'^2),
\eeq
where $K_M$ stands for the sectional curvature of $M^n$.
We will refer to \rf{NCCstrong} as the strong null convergence condition.
\begin{theorem}
\label{thfinal}
Let \grw\ be a spatially closed GRW spacetime obeying the strong null convergence condition, with $n\geq 3$.
Assume that $\m^n$ is a compact spacelike hypersurface immersed into \grw\ which is contained in a slab $\Omega(t_1,t_2)$
on which $f'$ does not vanish. If $H_{k}$ is constant, with $2\leq k\leq n$ then \m\ is totally umbilical. Moreover,
\m\ must be a slice $\{t_0\}\times M^n$ (necessarily with $f'(t_0)\neq 0$), unless in the case where \grw\ has positive
constant sectional curvature and \m\ is a round umbilical hypersphere.
The latter case cannot occur if we assume that the inequality in \rf{NCCstrong} is strict.
\end{theorem}
\begin{proof}
We may assume without loss of generality that $f'(h)>0$ on \m, and choose on \m\ the future-pointing Gauss map $N$. Then, we know from Lemma\rl{lemaellipticpoint} that there exists a point $p_0\in\m$ where all the principal curvatures (with respect to the chosen orientation) are negative. In particular, the constant $H_k$ is positive and we can consider the function
$\phi=H_k^{1/k}g(h)+\g{N}{K}\in\mathcal{C}^\infty(\m)$. Since $H_k$ is constant, from \rf{Lk(g)} and \rf{Lk(NK)2} we have
that
\begin{eqnarray}
\label{in5}
L_{k-1}\phi & = & k\bin{n}{k}(H_k-H_k^{1/k}H_{k-1})f'(h)\\
\nonumber {} & {} & +\bin{n}{k}\g{N}{K}(nH_1H_{k}-(n-k)H_{k+1}-kH_k^{(k+1)/k}) \\
\nonumber {} & {} & +\g{N}{K}\Theta,
\end{eqnarray}
where
\begin{eqnarray*}
\Theta & = & \frac{1}{f^2(h)}\sum_{i=1}^{n}\mu_{i,k-1}K_{M}(N^*\wedge E_i^*)\|N^*\wedge E_i^*\|^2\\
{} & {} & -(\log f)''(h)\left(c_{k-1}H_{k-1}\|\nabla h\|^2-\g{P_{k-1}\nabla h}{\nabla h}\right),
\end{eqnarray*}
and $P_{k-1}E_i=\mu_{i,k-1}E_i$ for every $i=1,\ldots,n$.

Since \m\ has an elliptic point, from the proofs of Theorem\rl{thbis} and Theorem\rl{thbiss}, respectively,  we know that inequalities \rf{garding} and \rf{garding2} hold on \m\ for every $k$ (inequality \rf{garding2} was derived in the proof of Theorem\rl{thbiss}, where we assumed that $k\leq n-1$; however \rf{garding2} trivially holds also for $k=n$ since $H_{n+1}=0$ by definition). From \rf{garding} we get that
\beq
\label{in1}
H_k-H_k^{1/k}H_{k-1}=H_k^{1/k}(H_k^{(k-1)/k}-H_{k-1})\leq 0
\eeq
on \m, with equality if and only if \m\ is totally umbilical. Using \rf{garding2} we also get that
\beq
\label{in2}
nH_1H_{k}-(n-k)H_{k+1}-kH_k^{(k+1)/k}\geq kH_k(H_1-H_k^{1/k})\geq 0
\eeq
on \m. On the other hand, Lemma\rl{lemafernanda} (when $k=2$) and Lemma\rl{lemaCR} (when $k\geq 3$) apply here and imply that the operator $L_{k-1}$ is elliptic or, equivalently, $P_{k-1}$ is positive definite. In particular, its eigenvalues
$\mu_{i,k-1}$ are all positive on \m, and from \rf{NCCstrong} we have
\beq
\label{in3}
\mu_{i,k-1}K_{M}(N^*\wedge E_i^*)\|N^*\wedge E_i^*\|^2\geq
\mu_{i,k-1}\alpha\|N^*\wedge E_i^*\|^2
\eeq
for every $i=1,\ldots,n$, where we are writing $\alpha=\sup_I(ff''-f'^2)$ for simplicity. From the decompositions
\[
N=N^*-\g{N}{\partial_t}\partial_t, \quad
E_i=E_i^*-\g{E_i}{\partial_t}\partial_t, \quad \mathrm{and} \quad
\partial_t=-\nabla h-\g{N}{\partial_t}N,
\]
it follows easily that
\[
\|N^*\wedge E_i^*\|^2=\|\nabla h\|^2-\g{E_i}{\nabla h}^2.
\]
Therefore, \rf{in3} implies
\begin{eqnarray*}
\sum_{i=1}^{n}\mu_{i,k-1}K_{M}(N^*\wedge E_i^*)\|N^*\wedge E_i^*\|^2 & \geq &
\alpha\left(\mathrm{tr}(P_{k-1})\|\nabla h\|^2-\sum_{i=1}^{n}\mu_{i,k-1}\g{E_i}{\nabla h}^2\right) \\
{} & = & \alpha\left(c_{k-1}H_{k-1}\|\nabla h\|^2-\g{P_{k-1}(\nabla h)}{\nabla h}\right),
\end{eqnarray*}
and then
\beq
\label{in4}
\Theta\geq
\left(\frac{\alpha}{f^2(h)}-(\log f)''(h)\right)
\left(c_{k-1}H_{k-1}\|\nabla h\|^2-\g{P_{k-1}(\nabla h)}{\nabla h}\right)\geq 0.
\eeq
For the last inequality, observe that $\alpha/f^2(h)-(\log f)''(h)\geq 0$ holds trivially, and also observe that
\begin{eqnarray*}
c_{k-1}H_{k-1}\|\nabla h\|^2-\g{P_{k-1}(\nabla h)}{\nabla h} & = &
\mathrm{tr}(P_{k-1})\|\nabla h\|^2-\g{P_{k-1}(\nabla h)}{\nabla h}\geq 0
\end{eqnarray*}
holds because $P_{k-1}$ is positive definite.

Summing up, using \rf{in1}, \rf{in2} and \rf{in4}, and taking into account that $f'(h)>0$ and $\g{N}{K}<0$, we obtain from \rf{in5} that $L_{k-1}\phi\leq 0$ on \m. Since $L_{k-1}$ is an elliptic operator on the Riemannian manifold \m, which is compact, we have, by the maximum principle, that $\phi$ must be constant. Hence,
$L_{k-1}\phi=0$ and the three terms in \rf{in5} vanish on \m. In particular, \rf{in1} is an equality and \m\ is a totally umbilical hypersurface. Since $H_k$ is a positive constant and \m\ is totally umbilical, we have that all the higher order mean curvatures are constant. In particular, $H_1$ is a positive constant and the result follows by Theorem\rl{th22}.
\end{proof}

\section*{Acknowledgements}
This work was started while the first author was visiting the Departamento de
Matem\'{a}tica of the Universidade Federal do Cear\'{a}, Fortaleza, Brazil. It was finished while the first author was visiting the Institute des Hautes \'{E}tudes Scientifiques (IH\'{E}S) in Bures-sur-Yvette, France, and the second author was
visiting the Institut de Mathematiques de Jussieu in Paris, France. The authors would like
to thank those institutions for their hospitality.


\bibliographystyle{amsplain}

\end{document}